\documentclass[12pt]{article}

\usepackage{amssymb}
\usepackage[T1]{fontenc}
\usepackage[utf8]{inputenc}
\usepackage[intlimits]{amsmath}
\usepackage{enumerate}
\usepackage[all]{xy}

\usepackage{amsthm}

\newtheorem{theorem}{Theorem}
\newtheorem{corollary}{Corollary}

\usepackage{multirow}

\begin{document}

\begin{center}\Large\bf
Complete and vertical lifts of Poisson vector fields and infinitesimal deformations  of Poisson tensor
\end{center}
\begin{center}\large
Alina Dobrogowska, Grzegorz Jakimowicz, Karolina Wojciechowicz
\end{center}

\begin{center}
Institute of Mathematics, University of Białystok, Ciołkowskiego 1M, 15-245 Białystok, Poland
\end{center}
\begin{center}
E-mail: alina.dobrogowska@uwb.edu.pl, g.jakimowicz@uwb.edu.pl, kzukowska@math.uwb.edu.pl
\end{center}







\begin{abstract}
In this paper we prove that both complete and vertical lifts of
a Poisson vector field from a Poisson manifold $(M, \pi)$ to its tangent bundle 
$(TM, \pi_{TM})$ are also Poisson. We use this fact to describe the infinitesimal deformations of Poisson tensor $\pi_{TM}$. We study some of their properties and present a extensive set of examples in a low dimensional case.
\end{abstract}

{\bf Keywords:}
Lie algebroid, linear Poisson structure, tangent and vertical  lifts of vector fields,  bi-Hamiltonian structure, Lie algebra,  tangent lift of Poisson structure


\section{Introduction}

The aim of this paper is to present the extension and generalization of results obtained in the articles \cite{AJ, AJK}.
We want to describe explicitly how to build some infinitesimal deformations of Poisson tensors generated by the algebroid structure of differential forms  using complete and vertical lifts of Poisson vector fields.

The paper is organized as follow. In the Section 2 
we recall such concepts, definitions and well known results from the Poisson geometry as the Poisson vector field, the Schouten--Nijenhuis bracket of vector fields, the bi-Hamiltonian structure, the Poisson cohomology, the infinitesimal deformations of Poisson tensors, see \cite{ a40, 5, CrFe, Zu,  Lich,  14, 10, Tsi-2, Wein}.
We also introduce some notions and results of the theory of Lie algebroids including related Poisson structures on the dual bundle to the Lie algebroid.
As reference to this material we recommend  \cite{5, 13, Ka, 2, 4, We}.

Section 3 contains main results of the paper. Using the concept of lifting multivectors from $M$ to $TM$, see \cite{GraUrb, Mba, MitVai, Ya}, we lift Poisson vector fields from $M$ to 
Poisson vector fields on $TM$ considered as a dual bundle to the algebroid $T^*M$.
The complete and vertical lifts of Poisson vector fields will be the building blocks for generating new Poisson structures on $TM$.
These structures are compatible with the original algebroid structure and can be considered as infinitesimal deformations. Some of these structures belong to deformation of cohomology of Lie algebroid (the fiber-wise linear Poisson cohomology), see also \cite{Cra-Moe}.
The next part of the paper is devoted to the examples of main theorems.

\section{Preliminaries and notations}

We give here a short review of some basic notions and facts in  Poisson geometry
which will be used in the main part of the paper.

Let $(M,\pi)$ be $N$-dimensional Poisson manifold, where $\pi\in \Gamma^{\infty} \left(\bigwedge^2 TM\right)$ is a Poisson tensor.
The Poisson bracket on $M$ is given by $\{f,g\}=\pi(df,dg)$ and it is 
a skew-symmetric bilinear mapping satisfying the Jacobi identity
and the Leibniz rule.
In a system of local coordinates ${\bf x}=(x^1,\dots, x^N)$ on $M$  it can be written in the form
\begin{equation}
\{f,g\}({\bf x})=\sum_{ i,j=1}^{N}\pi^{ij}({\bf x})\dfrac{\partial f}{\partial x^i} \dfrac{\partial g}{\partial x^j},
\end{equation}
where $\pi^{ij}({\bf x})=-\pi^{ji}({\bf x})=\{x^i,x^j\}$.

We denote by $\mathcal{X}(M)$ the space of smooth vector fields on a manifold $M$.
The Leibniz identity means that the map $f\longmapsto  \{f,h\}$ is a derivation for all $h
\in C^\infty(M)$. 
Thus, there is a unique vector field on $M$ called the Hamiltonian vector field of $h$, such that $C^{\infty}(M)\ni h\longmapsto X_h=\{\cdot, h\}\in \mathcal{X}(M)$. It means that a Poisson bracket defines an antihomomorphism $[X_{h_1},X_{h_2}]=-X_{\{h_1,h_2\}}$. 
The functions $c_i\in C^{\infty}(M)$ for which the Hamiltonian vector field $X_h$ vanishes identically are called Casimir functions.

Generally a vector field $X\in \mathcal{X}(M)$ on a Poisson manifold $M$ such that
\begin{equation}
\mathcal{L}_{X}\pi=0
\end{equation}
is called a Poisson vector field. It can be viewed as an infinitesimal automorphism of the Poisson structure.
The Lie derivative of the bi-vector $\pi$ along $X$ is given by the formula 
\begin{equation}
\left(\mathcal{L}_{X}\pi\right)(df,dg) = \mathcal{L}_{X}\left(\pi(df,dg)\right) -\pi (\mathcal{L}_{X}(df), dg)-\pi (df, \mathcal{L}_{X}(dg)),\end{equation} 
or equivalently in the terms of the Poisson bracket
\begin{equation}
\left(\mathcal{L}_{X}\pi\right)(df,dg) =  X\left(\{f,g\}\right)-\{X\left(f\right),g\}-\{f, X\left(g\right)\},
\end{equation}
where $f,g\in C^{\infty}(M)$.
Then the condition on the vector field $X$ to be Poisson is
\begin{equation}
\label{con}
X\left(\{f,g\}\right)-\{X\left(f\right),g\}-\{f, X\left(g\right)\}=0.
\end{equation}
In local coordinates when $X=\sum_{i=1}^{N}v^i\dfrac{\partial}{\partial x^i}$ we  rewrite it in the form
\begin{equation}
\sum_{s=1}^{N}\left(\dfrac{\partial \pi^{ij}}{\partial x^s} v^s-\pi^{sj}\dfrac{\partial v^i}{\partial x^s}-\pi^{is}\dfrac{\partial v^j}{\partial x^s} \right)=0.
\end{equation}
It can be seen that Hamiltonian vector fields are always Poisson as the equation (\ref{con}) reduces to the Jacobi identity.
Thus it is a property of the Poisson tensor that Lie derivative of $\pi$ with respect to $X_h$ vanishes $\mathcal{L}_{X_h}\pi=0$.
The Hamiltonian vector fields are a Lie algebra ideal in the Lie algebra of Poisson vector fields.

Next, we denote by $\mathcal{X}^k(M)=\Gamma^{\infty}\left(\Lambda^k TM\right)$ the space  of  $k$-vector fields on $M$. In addition, we  recall the definition  of the Schouten--Nijenhuis bracket, see \cite{8, 7}, which is a bilinear map 
$\mathcal{X}^k(M) \times \mathcal{X}^l(M)\ni(X,Y)\longmapsto  [X, Y]\in \mathcal{X}^{k+l-1}(M)$ 
determined by the properties
\begin{align}
\label{p6}
& [X, Y]=-(-1)^{(k-1)(l-1)}[Y, X],\quad X\in \mathcal{X}^k(M), Y\in \mathcal{X}^l(M),\\
& [X, Y\wedge Z]=[X, Y]\wedge Z +(-1)^{(k-1)l} Y\wedge [X, Z],\quad X\in \mathcal{X}^k(M),\nonumber\\
&\quad\quad\quad\quad\quad\quad\quad Y\in \mathcal{X}^l(M), Z\in \mathcal{X}^p(M), \nonumber\\
& [X,Y]= \mathcal{L}_{X}Y, ,\quad X\in \mathcal{X}(M), Y\in \mathcal{X}^l(M),\nonumber
\end{align}
where $[X,Y]$  is the commutator bracket of vector fields, $[X,f]=X(f)$ and $[f,g]=0$ for $X,Y\in \mathcal{X}(M)$,  $f,g\in \mathcal{X}^{0}(M)= C^{\infty}(M)$.
The Schouten bracket satisfies the graded Jacobi identity
\begin{equation}
[X,[Y,Z]]=[[X,Y],Z]+(-1)^{(k-1)(l-1)}[Y,[X,Z]]
\end{equation}
for $X\in \mathcal{X}^k(M)$, $Y\in \mathcal{X}^l(M)$, $Z\in \mathcal{X}^p(M)$.
Thus $\left(\bigoplus_{k=0}^{N} \mathcal{X}^k(M), [\cdot, \cdot]\right)$ is a graded Lie algebra.

We say that two Poisson tensors $\pi_1$ and $\pi_2$ are compatible if the linear combination $\pi_1+\lambda\pi_2$, $\lambda\in\mathbb{R}$, is again a Poisson structure. It is equivalent to the vanishing of Schouten--Nijenhuis bracket $[\pi_1, \pi_2]=0$. In particular $[\pi, \pi]=0$ is the Jacobi identity.
The manifold $M$ equipped with two compatible Poisson structures $\pi_1, \pi_2$ is called a bi-Hamiltonian manifold, see \cite{a39,a35,a19}.

The concept of a Poisson cohomology was first introduced by Lichnerowicz \cite{Lich}.
For a Poisson tensor $\pi$ the $\mathbb{R}$-linear map $\delta_{\pi}:\mathcal{X}^k(M)\longrightarrow \mathcal{X}^{k+1}(M)$
defined by     the Schouten--Nijenhuis bracket 
\begin{equation}
\delta_{\pi}(X)=[\pi, X]
\end{equation}
is a coboundary, i.e. $\delta_{\pi}\circ \delta_{\pi}=0$.
This is the consequence of the property $[\pi,[\pi,\cdot]]=0$. The operator $\delta_{\pi}$ is called the Lichnerowicz--Poisson
differential.
This mapping generates the Lichnerowicz complex of multivector fields
\begin{equation}
\cdots \stackrel{\delta_{\pi}}{\longrightarrow} \mathcal{X}^k(M) \stackrel{\delta_{\pi}}{\longrightarrow}\mathcal{X}^{k+1}(M) 
\stackrel{\delta_{\pi}}{\longrightarrow} \mathcal{X}^{k+2}(M) 
\stackrel{\delta_{\pi}}{\longrightarrow}\cdots .
\end{equation}
The  cohomology of this complex $(\mathcal{X}^{*}(M),\delta_{\pi})$  is called Poisson cohomology. For symplectic manifolds,  the Poisson cohomology is  isomorphic to the de Rham cohomology. In the special case when we have linear Poisson structures the Poisson cohomology is related to Lie algebra cohomology (Chevalley--Eilenberg cohomology).
The Poisson cohomology groups are denoted by $H^{*}_{\pi}(M)=\bigoplus_{k=0}^{\infty} H^{k}_{\pi}(M)$, where
\begin{equation}
H^{k}_{\pi}(M)=\dfrac{ker \left(\delta_{\pi}: \mathcal{X}^k(M)\longrightarrow \mathcal{X}^{k+1}(M)\right) }{Im \left(\delta_{\pi}: 
\mathcal{X}^{k-1}(M)\longrightarrow \mathcal{X}^k(M)\right)}.
\end{equation}
In particular, the zero Poisson cohomology group $H^{0}_{\pi}(M)$ coincides with the ring of Casimir functions of $\pi$, i.e. $[\pi, c]=\{\cdot
,c\}=0$.
The first Poisson cohomology group $H^{1}_{\pi}(M)$ is the quotient of the Lie algebra of infinitesimal symmetries (space of Poisson vector fields), i.e. $[\pi,X]=-\mathcal{L}_{X}\pi=0$, over the space of Hamiltonian vector fields, i.e. $[\pi,h]=X_h$. It is called the space of outer automorphisms of Poisson manifold.
Next the  Poisson cohomology group $H^{2}_{\pi}(M)$ is the quotient of the space of bi-vector fields $X$ which  satisfy the condition $[\pi,X]=0$  over the space of  bi-vector fields which can be presented in the form  $X=[\pi, Y]$, where $Y\in \mathcal{X}(M) $.
As reference to this material, we recommend \cite{Am, Cra-Moe, Zu, Ho}.

The Poisson cohomology is a useful tool in Poisson geometry, it plays an important
role  in  deformation theory and gives some information about the geometry of the manifold.
If we consider a formal one-parameter deformation of a Poisson
structure $\pi$ given by
\begin{equation}
\pi_{\lambda}=\pi+\lambda X,
\end{equation}
where $X\in \mathcal{X}^2(M)$, $\lambda\in\mathbb{R}$, then the condition for $\pi_{\lambda}$ to be a Poisson tensor gives
\begin{align}
[\pi_{\lambda},\pi_{\lambda}] & =[\pi+\lambda X,\pi+\lambda X]=2\lambda[\pi,X]+\lambda^2[X,X]=0.
\end{align}
If $X\in H^{2}_{\pi}(M)$ then 
\begin{equation}
[\pi+\lambda X,\pi+\lambda X]=\lambda^2[X,X]
\end{equation}
 satisfies the Jacobi identity up to terms of order $\lambda^2$.
So $\pi+\lambda X^2$ is called an infinitesimal deformation  of Poisson tensor $\pi$. Thus $H^{2}_{\pi}(M)$ is interpreted as nontrivial infinitesimal deformations.

Next we recall that a Lie algebroid $(A,[\cdot, \cdot]_{A},a)$ over manifold $M$ is a vector bundle $q_M:A\longrightarrow M$ together with
a vector bundle map $a:A\longrightarrow   TM $, called the anchor, and a Lie bracket $[\cdot, \cdot]_{A}$ on the space of sections $\Gamma(A)$. A Lie bracket $[\cdot, \cdot]_{A}:\Gamma(A)\times \Gamma(A)\longrightarrow \Gamma(A)$ satisfies the following Leibniz rule
\begin{equation}
[X_{A}, fY_{A}]=f[X_{A}, Y_{A}]+a(X_A)(f)Y_A
\end{equation}
for any sections $X_A,Y_A\in \Gamma(A)$ and function $f\in C^{\infty}(M)$.
The anchor meets the condition
\begin{equation}
a\left([X_{A}, Y_{A}]\right)=[a\left(X_{A}\right),a\left(Y_{A}\right) ].
\end{equation}
This notion was introduced by Pradines \cite{Par}, see also \cite{2,4}.
It is well-known that a Poisson structure on $M$ induces the algebroid structure on the cotangent bundle $A=T^{*}(M)$ of $M$ by the property
\begin{align}
[df,dg]=d\{f,g\},\\
a(df)(g)=\{f,g\}
\end{align}
for $f,g\in C^{\infty}(M)$, see \cite{Ma-Mor}.

There is a natural correspondence between Lie algebroids $A$ and Poisson structures which are linear on the fibers on the total space of the dual bundle $A^*$ of $A$. For instance, if $A=T^{*}(M)$ 
\begin{equation}
 \label{d1111xxx}
\xymatrix{
& & A=T^*M \ar@<.0ex>[dd]^*+<1ex>\txt{\tiny{{$q^*_{_{M}}$}}} \ar@<-.0ex>[rr]^*+<1ex>\txt{\tiny{{$a$}}}  & & TM \ar@<.0ex>[dd]^*+<1ex>\txt{\tiny{{$q_{_{M}}$}}}&&\\
    & &&&   & &\\
  && M  \ar@<-.0ex>[rr]^*+<1ex>\txt{\tiny{{$id$}}}  & & M &&\\
 }
 \\
 \end{equation}
then the Poisson bracket on $A^{*}=TM$ is given by the properties
\begin{align}
& \label{1}\{f\circ q_M, g\circ q_M\}_{TM}=0,\\
& \label{2}\{f\circ q_M, l_{dg}\}_{TM}=-a(dg)(f)\circ q_M,\\
& \{ l_{df}, l_{dg}\}_{TM}=l_{[df,dg]}
\end{align}
for functions $f,g\in C^{\infty}(M)$ on the base and fiber-wise linear functions $l_{df},l_{dg}\in C^{\infty}(TM)$, where $l_{df}$ is given by pairing
\begin{equation}
l_{df}(X)=\big< X, df(q_M(X))\big>, \quad \forall X\in TM.
\end{equation}
We denote by $y^i=l_{dx^i}$.
This type of Poisson tensors is called a fiber-wise linear Poisson structure, see \cite{Zu}.
The Poisson tensor can be rewritten in the form
\begin{equation}
\pi_{TM}({\bf x}, {\bf y})=\sum_{i,j=1}^{N} \pi^{ij}({\bf x})\frac{\partial}{\partial x^i}\wedge \frac{\partial}{\partial y^j}+
\dfrac{1}{2}\sum_{i,j,s=1}^{N}\dfrac{\partial\pi^{ij} }{\partial x^s}({\bf x})y^s
\frac{\partial}{\partial y^i}\wedge \frac{\partial}{\partial y^j}
\end{equation}
or presented graphically
\begin{equation}
\label{tensor-pi}
\pi_{TM}({\bf x}, {\bf y})=\left(
\begin{array}{c|c}
0& \pi({\bf x})\\
\hline
\pi({\bf x}) & \sum_{s=1}^{N}\dfrac{\partial\pi }{\partial x^s}({\bf x})y^s
\end{array}
\right),
\end{equation}
where $({\bf x}, {\bf y})=(x^1,\dots, x^N, y^1,\dots, y^N)$ is a system of local coordinates on $TM$.
Some of the properties of this Poisson structure are well known, see \cite{AJ, GraUrb}.
If $c_1,\dots, c_r$ are  Casimir functions for the Poisson structure $\pi$,  then the functions 
\begin{equation}
\label{cas}
c_i\circ q_M \quad \textrm{and}\quad l_{dc_i}=\sum_{s=1}^{N}\dfrac{\partial c_i}{\partial x^s}({\bf x})y^s, \quad i=1,\dots r,
\end{equation}
are Casimir functions for the Poisson tensor $\pi_{TM}$. 
Subsequently if the functions $\{H_i\}_{i=1}^k$ are in involution with respect to the Poisson tensor $\pi$,
then the functions 
\begin{equation}
\label{cas-1-n}
H_i\circ q_M \quad \textrm{and}\quad  l_{dH_i}=\sum_{s=1}^{N}\dfrac{\partial H_i}{\partial x^s}({\bf x})y^s,\quad i=1,\dots k,
\end{equation}
 are in involution with respect to the Poisson tensor $\pi_{TM}$ given by (\ref{tensor-pi}).

Cranic and Moerdijk \cite{Cra-Moe} introduced the deformation cohomology for Lie algebroids.
In the next section we show how to lift the Poisson cohomology on $M$ to a fiber-wise linear Poisson cohomology on $A^{*}=TM$.

\section{Lift of Poisson vector fields and some infinitesimal deformations of Poisson tensors}

In the beginning we recall the known facts about lift of multivector fields from manifold $M$ to $TM$, see \cite{GraUrb, Mba, Ya}.
Given a $k$-vector field in local coordinates 
\begin{equation}
X=\sum_{i_1,\dots , i_k=1}^{N} v^{i_1\dots i_k}({\bf x})\dfrac{\partial}{\partial x^{i_1}}\wedge \dots \wedge \dfrac{\partial}{\partial x^{i_k}}
\in \mathcal{X}^{k}(M)
\end{equation}
we have the following complete lift to $\mathcal{X}^{k}(TM)$
$$
X_{C}=\sum_{i_1,\dots, i_k=1}^{N} \left( v^{i_1\dots i_k}({\bf x})\dfrac{\partial}{\partial y^{i_1}}\wedge \dots \wedge \dfrac{\partial}{\partial y^{i_{l-1}}}\wedge \dfrac{\partial}{\partial x^{i_l}} \wedge \dfrac{\partial}{\partial y^{i_{l+1}}} \wedge \dots\dfrac{\partial}{\partial y^{i_k}}+\right.
$$
\begin{equation}
\left. +\sum_{s=1}^{N} \dfrac{\partial v^{i_1\dots i_k}}{\partial x^{s}}({\bf x})y^s
\dfrac{\partial}{\partial y^{i_1}}\wedge \dots \wedge \dfrac{\partial}{\partial y^{i_k}}\right).
\end{equation}
The vertical lift of $X$ from $M$ to $TM$ we denote by $X_{V}$ and it is defined by
\begin{equation}
X_{V}=\sum_{i_1,\dots , i_k=1}^{N} v^{i_1\dots i_k}({\bf x})\dfrac{\partial}{\partial y^{i_1}}\wedge \dots \wedge \dfrac{\partial}{\partial y^{i_k}}\in \mathcal{X}^{k}(TM).
\end{equation}
For $X\in \mathcal{X}^{k}(M)$ and $Y\in \mathcal{X}^{l}(M)$ we obtain the following commutator relations for $X_{C},X_{V}\in \mathcal{X}^{k}(TM)$, $Y_{C},Y_{V}\in \mathcal{X}^{l}(TM)$
\begin{align}
& [X_{C},Y_{C}]=[X,Y]_{C},\\
& [X_{C},Y_{V}]=[X,Y]_{V},\\
\label{3-c}& [X_{V},Y_{V}]=0.
\end{align}

Using the above formulas we lift  Poisson vector fields on $M$ to  Poisson vector fields on $TM.$
\begin{theorem}
If $X=\sum_{i=1}^N v^i({\bf x})\frac{\partial}{\partial x^i}$ is a Poisson vector field on a Poisson manifold $(M,\pi)$ then
\begin{equation}
\label{psvf1}
X_{C}=\sum_{i=1}^N v^i({\bf x})\frac{\partial}{\partial x^i}+\sum_{i,s=1}^N \frac{\partial v^i}{\partial x^s}({\bf x})y^s\frac{\partial}{\partial y^i},
\end{equation}
\begin{equation}
\label{psvf2}
X_{V}=\sum_{i=1}^N v^i({\bf x})\frac{\partial}{\partial y^i}
\end{equation}
are Poisson vector fields on the Poisson manifold $(TM,\pi_{TM}).$
\end{theorem} 
\begin{proof}
To check this, it is enough to check that it holds on local system of coordinates $(x^1,\dots, x^N,y^1,\dots,y^N)$.
By direct calculation we obtain
\begin{align}
X_{C}\left(\{x^i,x^j\}_{TM}\right)&-\{X_{C}(x^i),x^j\}_{TM}-\{x^i,X_{C}(x^j)\}_{TM}=\\
&=-\{v^i({\bf x}),x^j\}_{TM}-\{x^i,v^j({\bf x})\}_{TM}=0,\nonumber
\end{align}
\begin{align}
X_{C}\left(\{x^i,y^j\}_{TM}\right)&-\{X_{C}(x^i),y^j\}_{TM}-\{x^i,X_{C}(y^j)\}_{TM}=\\
&=X_{C}\left(\pi^{ij}({\bf x})\right)-\{v^i({\bf x}),y^j\}_{TM}-\{x^i,\sum_{s=1}^N\frac{\partial v^j}{\partial x^s}({\bf x})y^s\}_{TM}=\nonumber\\
&=\sum_{s=1}^N\left(v^s({\bf x})\frac{\partial \pi^{ij}}{\partial x^s}({\bf x})-\pi^{sj}({\bf x})\frac{\partial v^i}{\partial x^s}({\bf x})-\pi^{is}({\bf x})\frac{\partial v^j}{\partial x^s}({\bf x})\right)=\nonumber\\
&=\left(\mathcal{L}_X\pi\right)\left(dx^i,dx^j\right)=0,\nonumber
\end{align}
\begin{align}
& X_{C}\left(\{y^i,y^j\}_{TM}\right)-\{X_{C}(y^i),y^j\}_{TM}-\{y^i,X_{C}(y^j)\}_{TM}=\\
&=X_{C}\left(\sum_{s=1}^N\frac{\partial\pi^{ij}}{\partial x^s}({\bf x})y^s\right)-\{\sum_{s=1}^N\frac{\partial v^i}{\partial x^s}({\bf x})y^s,y^j\}_{TM}-\{y_i,\sum_{s=1}^N\frac{\partial v^j}{\partial x^s}({\bf x})y^s\}_{TM}=\nonumber\\
&=\sum_{m,s=1}^N y^m\frac{\partial}{\partial x^m}\left(v^s({\bf x})\frac{\partial \pi^{ij}}{\partial x^s}({\bf x})-\pi^{sj}({\bf x})\frac{\partial v^i}{\partial x^s}({\bf x})-\pi^{is}({\bf x})\frac{\partial v^j}{\partial x^s}({\bf x})\right)=\nonumber\\
&=\sum_{m,s=1}^N y^m\frac{\partial}{\partial x^m}\left(\left(\mathcal{L}_X\pi\right)\left(dx^i,dx^j\right)\right)=0.\nonumber
\end{align}
Performing the similar calculation for $X_{V}$ 
\begin{align}
X_{V}\left(\{x^i,x^j\}_{TM}\right)&-\{X_{V}(x^i),x^j\}_{TM}-\{x^i,X_{V}(x^j)\}_{TM}=0
\end{align}
\begin{align}
X_{V}\left(\{x^i,y^j\}_{TM}\right)&-\{X_{V}(x^i),y^j\}_{TM}-\{x^i,X_{V}(y^j)\}_{TM}=\\
&=X_{V}\left(\pi^{ij}({\bf x})\right)-\{x^i, v^j({\bf x})\}_{TM}=0 ,\nonumber
\end{align}
\begin{align}
X_{V}\left(\{y^i,y^j\}_{TM}\right)&-\{X_{V}(y^i),y^j\}_{TM}-\{y^i,X_{V}(y^j)\}_{TM}=\\
&=X_{V}\left(\sum_{s=1}^N\frac{\partial\pi^{ij}}{\partial x^s}({\bf x})y^s\right)-\{ v^i({\bf x}),y^j\}_{TM}-\{y_i, v^j({\bf x})\}_{TM}=\nonumber\\
&=\sum_{s=1}^N \left( v^s({\bf x})\frac{\partial \pi^{ij}}{\partial x^s}({\bf x})-\pi^{sj}({\bf x})\frac{\partial v^i}{\partial x^s}({\bf x})-\pi^{is}({\bf x})\frac{\partial v^j}{\partial x^s}({\bf x})\right)=\nonumber\\
&=\sum_{s=1}^N \left(\mathcal{L}_X\pi\right)\left(dx^i,dx^j\right)=0.\nonumber
\end{align}
We see that these are also  Poisson vector fields on $\left(TM,\pi_{TM}\right)$.
\end{proof}
Above, the first vector field, given by $(\ref{psvf1})$, is a fiber--wise linear vector field and the second, given by $(\ref{psvf2})$, is a fiber--wise constant vertical vector field.

\begin{theorem}
If $X_{C}$, $X_{V}$ are complete and vertical lifts of vector field $X\in \mathcal{X}(M)$, then
the bi-vector
\begin{equation}\label{def}
\pi_{X_{C},X_{V}}=X_{C}\wedge X_{V}
\end{equation}
is a Poisson tensor on $TM$.
\end{theorem}

\begin{proof}
It is easy to see, from definition, that (\ref{def}) is antisymmetric. If so, it is enough to check if the Jacobi identity holds. Direct calculation, using properties of the Schouten--Nijenhuis bracket given by (\ref{p6}) yields
\begin{align}
\left[\pi_{X_{C},X_{V}},\pi_{X_{C},X_{V}}\right]&=\left[X_{C}\wedge X_{V},X_{C}\wedge X_{V}\right]=\\
&=\left[X_{C}\wedge X_{V},X_{C}\right]\wedge X_{V}-X_{C}\wedge\left[X_{C}\wedge X_{V}, X_{V}\right]=\nonumber\\
&=-\left[X_{C},X_{C}\wedge X_{V}\right]\wedge X_{V}+X_{C}\wedge \left[X_{V},X_{C}\wedge X_{V}\right]=\nonumber\\
&=-\left[X_{C},X_{C}\right]\wedge X_{V}\wedge X_{V}-X_{C}\wedge \left[X_{C},X_{V}\right]\wedge X_{V}+\nonumber\\
&+X_{C}\wedge \left[X_{V},X_{C}\right]\wedge X_{V}+X_{C}\wedge X_{C}\wedge\left[X_{V}, X_{V}\right]=\nonumber\\
&=-2X_{C}\wedge \left[X_{C},X_{V}\right]\wedge X_{V}=2 \left[X,X\right]_{V}\wedge X_{C}\wedge  X_{V}=0.\nonumber
\end{align}
Then $\left[\pi_{X_{C},X_{V}},\pi_{X_{C},X_{V}}\right]=0.$ It means that $\pi_{X_{C},X_{V}}$ is a Poisson tensor.
\end{proof}

In a local system of coordinates $(x^1,\dots,x^N,y^1,\dots, y^N)$ the matrix of the Poisson tensor has the form
\begin{equation}
\begingroup\makeatletter\def\f@size{8}\check@mathfonts 
\pi_{X_{C},X_{V}}  ({\bf x}, {\bf y})=\left(
\begin{array}{c|c}
 0 & v({\bf x})v^{\top}({\bf x})\\
\hline
-v({\bf x})v^{\top}({\bf x}) & \sum_{s=1}^{N}\left( \dfrac{\partial v }{\partial x^s}({\bf x})v^{\top}({\bf x})-v({\bf x})\left( \dfrac{\partial v }{\partial x^s}({\bf x})\right)^{\top}\right)y^s
\end{array}
\right),
\endgroup
\end{equation}
where $v^{\top}=(v^1,\dots v^N).$

Note also that if $X$ is a Poisson vector field on a Poisson manifold $(M,\pi)$, then the Lie derivative of the Casimir function $c_i$ is again a Casimir function, i.e. 
\begin{equation}
\label{cas-3}
\mathcal{L}_{X}c_i=X(c_i)=c_j.
\end{equation}
In the case when we get zero, we have the following theorem.
\begin{theorem}
\label{th-333}
Let $X$ be a vector field on $M$ and let $f$ be smooth function on $M$. If $X(f)=0$ then for complete and vertical lifts of vector field $X$ we have
\begin{align}
& X_C(f\circ  q_M)=0, &  X_C(l_{df})=0,\\
& X_V(f\circ  q_M)=0, &  X_V(l_{df})=0.
\end{align}
\end{theorem}
\begin{proof}
In the beginning we assume that  for a initial vector field 
$X=\sum_{i=1}^{N}v^i\dfrac{\partial }{\partial x^i}$ we have $X(f)=\sum_{i=1}^{N}v^i\dfrac{\partial f}{\partial x^i}=0$ .
Next, a direct verification shows that
\begin{align}
 & X_C(f\circ  q_M)=X(f)=0,\\
 & X_V(f\circ  q_M)=\sum_{s=1}^{N}y^s\dfrac{\partial }{\partial x^s}X(f)=0,\\
 &  X_V(f\circ  q_M)=0,\\
 &  X_V(l_{df})=X(f)=0.
\end{align}
\end{proof}
Note that from the above theorem follows
that if $X(f)=0$ then $f$ is a Casimir function for a Poisson tensor $\pi_{X_{C},X_{V}}$.

\begin{theorem}
\label{theorem-3a}
Let $(M,\pi)$ be a Poisson manifold and let $c\in H_{\pi}^0(M)$ be a Casimir function for $\pi$. For each  Poisson vector field $X$ on $M$ 
\begin{equation}
\pi_{TM,X_{C},X_{V},c}=\pi_{TM}+\lambda c({\bf x}) \pi_{X_{C},X_{V}}
\end{equation}
is a Poisson tensor on $TM$.
\end{theorem}

\begin{proof}
From (\ref{p6}) we know that
\begin{align}
& \left[\pi_{TM,X_{C},X_{V},c},\pi_{TM,X_{C},X_{V},c}\right]=\left[\pi_{TM},\pi_{TM}\right]+\lambda \left[\pi_{TM},c({\bf x})X_{C}\wedge X_{V}\right] +\\
&+\lambda \left[c({\bf x})X_{C}\wedge X_{V},\pi_{TM}\right]+\lambda^2 c^2({\bf x})\left[X_{C}\wedge X_{V},X_{C}\wedge X_{V}\right]=\nonumber\\
&=2\lambda c({\bf x})\left[\pi_{TM}, X_{C}\right]\wedge X_{V}-2\lambda c({\bf x})X_{C}\wedge[\pi_{TM},X_{V}]=0.\nonumber
\end{align}
Above we use that $\pi_{TM}$ and $X_{C}\wedge X_{V}$ are Poisson tensors and $X_{C}, X_{V}$ are Poisson vector fields on $(TM,\pi_{TM}).$
\end{proof}
This is a consequence of the facts that both bi-vector fields are Poisson tensors and $\pi_{X_{C},X_{V}}\in H^2_{\pi_{TM }}(TM)$.
In local coordinates $(x^1,\dots,x^N,y^1\dots, y^N)$ we get the following infinitesimal deformation of the Poisson tensor $\pi_{TM}$
\begin{equation}
\label{ten-3} 
\pi_{TM,X_{C},X_{V},c}  ({\bf x}, {\bf y})=
\end{equation}
$$
\begingroup\makeatletter\def\f@size{8}\check@mathfonts
\left(
\begin{array}{c|c}
 0 & \pi ({\bf x})+ \lambda c({\bf x})v({\bf x})v^{\top}({\bf x})\\
\hline
\pi ({\bf x})- \lambda c({\bf x})v({\bf x})v^{\top}({\bf x}) & \sum_{s=1}^{N}\left(\dfrac{\partial\pi }{\partial x^s}({\bf x})+
\lambda c({\bf x})\left( \dfrac{\partial v }{\partial x^s}({\bf x})v^{\top}({\bf x})-v({\bf x})\left( \dfrac{\partial v }{\partial x^s}({\bf x})\right)^{\top}\right)\right)y^s
\end{array}
\right).
\endgroup
$$
This type of Poisson structures has already appeared in our article \cite{AJK}, where we described in detail the algebroid structure associated with it.
The next statement describes the case that Casimir functions for a certain class of tensors do not change.
\begin{theorem}
\label{c-1}
Let $c_1,\dots, c_r$ be  Casimir functions for the Poisson structure $\pi$ such that $\mathcal{L}_{X}c_i=0$. Then the functions 
\begin{equation}
\label{cas-1-1}
c_i\circ q_M \quad \textrm{and}\quad l_{dc_i}=\sum_{s=1}^{N}\dfrac{\partial c_i}{\partial x_s}({\bf x})y_s, \quad i=1,\dots r,
\end{equation}
are the Casimir functions for the Poisson tensor $\pi_{TM,X_{C},X_{V},c}$. 
\end{theorem}
\begin{proof}
This is the consequence of Theorem \ref{th-333} and formula (\ref{cas-3}).
\end{proof}

In the next step, we will consider a situation a little more general.
We will assume that we have two non-proportional Poisson vector fields $X,Y\in \mathcal{X}(M)$ on a Poisson manifold $(M,\pi)$ expressed in local coordinates as
\begin{equation}
X=\sum_{i=1}^N v^i({\bf x})\frac{\partial}{\partial x^i}, \quad Y=\sum_{i=1}^N w^i({\bf x})\frac{\partial}{\partial x^i}, 
\end{equation}
where $v^i, w^i\in C^{\infty}(M)$. Then we obtain four Poisson vector fields  on a Poisson manifold $(TM,\pi_{TM})$
\begin{align}
\label{vec}
& X_{C}=\sum_{i=1}^N v^i({\bf x})\frac{\partial}{\partial x^i}+\sum_{i,s=1}^N \frac{\partial v^i}{\partial x^s}({\bf x})y^s\frac{\partial}{\partial y^i}, &&\!\!\!
Y_{C}=\sum_{i=1}^N w^i({\bf x})\frac{\partial}{\partial x^i}+\sum_{i,s=1}^N \frac{\partial w^i}{\partial x^s}({\bf x})y^s\frac{\partial}{\partial y^i},\nonumber\\
& X_{V}=\sum_{i=1}^N v^i({\bf x})\frac{\partial}{\partial y^i}, &&
Y_{V}=\sum_{i=1}^N w^i({\bf x})\frac{\partial}{\partial y^i}.
\end{align}
In this case, we can  build three different types of bi-vector fields.
A detailed analysis of these cases will be presented in the following statements.

\begin{theorem}
\label{theorem-2}
If $X_{V}$ and $Y_{V}$ are given by (\ref{vec})  then
\begin{equation}
\label{con-1}
\pi_{X_{V}, Y_{V}}=X_{V}\wedge Y_{V}=\left( X\wedge Y\right)_{V}
\end{equation}
is a Poisson tensor on $TM$.
\end{theorem}

\begin{proof}
A direct calculation gives us
\begin{equation}
[\pi_{X_{V}, Y_{V}}, \pi_{X_{V}, Y_{V}}]=2[Y_{V}, X_{V}]\wedge X_{V}\wedge Y_{V}=0
\end{equation}
from (\ref{p6}) and (\ref{3-c}).
\end{proof}
This is a consequence of simple observation that   any structure of bi-vector field   with a matrix form
\begin{equation}
\left(
\begin{array}{c|c}
 0 & 0\\
\hline
0 & A({\bf x})
\end{array}
\right),
\end{equation}
where $A({\bf x})\in \mathfrak{so}(\frac{N}{2})$, is a Poisson tensor on $TM$.
Moreover, the Poisson tensor (\ref{con-1}) is compatible with the Poisson tensor $\pi_{TM}$.

\begin{theorem}
\label{VV}
Let $(M,\pi)$ be a Poisson manifold and let $c\in H_{\pi}^0(M)$ be a Casimir function for $\pi$. For any Poisson vector fields $X, Y$ on $M$
\begin{equation}
\label{ten-kon2}
\pi_{TM,X_{V},Y_{V},c}=\pi_{TM}+\lambda c({\bf x})X_{V}\wedge Y_{V}
\end{equation}
is a Poisson tensor on $TM$.
\end{theorem}

\begin{proof}
The statement follows from the facts that $\pi_{TM}$, $\pi_{X_{V}, Y_{V}}$ are the Poisson tensors and the bi-vector field (\ref{con-1}) belongs to the second Poisson cohomology group $\pi_{X_{V}, Y_{V}}\in H^2_{\pi_{TM}}(TM)$, i.e. $[\pi_{TM}, X_{V}\wedge Y_{V}]=0$.
\end{proof}
The local expression of (\ref{ten-kon2}) is given by
\begin{equation}
\begingroup\makeatletter\def\f@size{8}\check@mathfonts 
\pi_{T,X_{V},Y_{V},c}  ({\bf x}, {\bf y})=\left(
\begin{array}{c|c}
 0 & \pi ({\bf x})\\
\hline
\pi ({\bf x}) 
& \sum_{s=1}^{N}\dfrac{\partial\pi }{\partial x^s}({\bf x})y^s
+ \lambda c({\bf x})\left(v({\bf x})w^{\top}({\bf x})-w({\bf x})v^{\top}({\bf x})\right)
\end{array}
\right),
\endgroup
\end{equation}
where $v^{\top}=(v^1,\dots, v^N)$ and $w^{\top}=(w^1,\dots, w^N)$.

Additionally if $X\wedge Y$ is also a Poisson tensor on $M$ then this is the lift of bi-Hamiltonian structure $(M,\pi,X\wedge Y)$ to $TM$, see \cite{AJ}.
This is the case if the following condition is satisfied
\begin{equation}
\label{55}
[X\wedge Y, X\wedge Y]=2[X,Y]\wedge X\wedge Y=
\end{equation}
$$ =
\sum_{i,j,k,n=1}^{N}
v^j({\bf x}) w^k ({\bf x}) \left( v^n({\bf x})\dfrac{\partial w^i}{\partial x^n}({\bf x})-w^n({\bf x})\dfrac{\partial v^i}{\partial x^n}({\bf x})\right)
\dfrac{\partial}{\partial x^i} \wedge \dfrac{\partial}{\partial x^j} \wedge \dfrac{\partial}{\partial x^k}=0 .
$$
From the equality for Lichnerowicz--Poisson differential $\delta_{\pi_{TM}}\left(\left( X\wedge Y\right)_{V}\right)=\left(\delta_{\pi}\left( X\wedge Y\right)\right)_{V}$ the map $H_{\pi}^2(M)\ni X\wedge Y\mapsto \left( X\wedge Y\right)_{V}\in H^2_{\pi_{TM}}(TM)$ is a homomorphism of Poisson cohomology space, see \cite{Mba}.

\begin{theorem}
\label{theorem-last-V}
\begin{enumerate}
\item Let $X_i\in\mathcal{X}(M)$ for $i=1,2,3,4$  and let   $X_{i,V}$ be the  vertical lifts of the vector fields $X_i$ on $TM$.   Then
\begin{align}
\label{con-4-V}
& \pi_{X_{1,V}, X_{2,V},X_{3,V}, X_{4,V}}=X_{1,V}\wedge X_{2,V}+X_{3,V}\wedge X_{4,V}
\end{align}
is the   Poisson tensor on $TM$.
\item 
Let $(M,\pi)$ be a Poisson manifold and let $c\in H_{\pi}^0(M)$ be a Casimir function for $\pi$.
For any Poisson vector fields  $X_i $ for $i=1,2,3,4$ on $M$ 
\begin{equation}
\label{ten-kon2-V}
\pi_{TM,X_{1,V}, X_{2,V},X_{3,V}, X_{4,V},c}=\pi_{TM}+\lambda c({\bf x})\left( X_{1,V}\wedge X_{2,V}+X_{3,V}\wedge X_{4,V}\right)
\end{equation}
is a Poisson tensor on $TM$.
\end{enumerate}
\end{theorem}

\begin{proof}
Proof  is obtained by direct calculation.
\end{proof}
{\bf Remark:} The above procedure can be repeated many times.

\begin{theorem}
\label{theorem-3}
Let $X,Y\in\mathcal{X}(M)$ be such that $[X,Y]=0$ and let $X_{C}$,  $Y_{C}$, $Y_{V}$ be the complete and vertical lifts of the vectors $X$, $Y$ on $TM$, respectively.   Then
\begin{align}
\label{con-3}
& \pi_{X_{C}, Y_{V}}=X_{C}\wedge Y_{V},\\
\label{con-4}
& \pi_{X_{C}, Y_{C}}=X_{C}\wedge Y_{C}
\end{align}
are  Poisson tensors on $TM$.
\end{theorem}

\begin{proof}
After direct calculation of the Jacobi identity using  the Schouten--Nijenhuis bracket  we obtain
\begin{align}
&[\pi_{X_{C}, Y_{V}}, \pi_{X_{C}, Y_{V}}]=2[X_{C},Y_{V}]\wedge X_{C}\wedge Y_{V} =2[X,Y]_{V}\wedge X_{C}\wedge Y_{V}=0.
\end{align}
Similarly for the second construction
\begin{align}
&[\pi_{X_{C}, Y_{C}},\pi_{X_{C}, Y_{C}}]=2[X_{C}, Y_{C}]\wedge X_{C} \wedge Y_{C}=2[X, Y]_{C}\wedge X_{C} \wedge Y_{C}=0.
 \end{align}
\end{proof}

\begin{theorem}
\label{C}
Let $(M,\pi)$ be a Poisson manifold and let $c\in H_{\pi}^0(M)$ be a Casimir function for $\pi$. For any Poisson vector fields $X, Y$ on $M$ such that $[X,Y]=0$
\begin{align}
\label{ten-kon2-C}
& \pi_{TM,X_{C},Y_{V},c}=\pi_{TM}+\lambda c({\bf x})X_{C}\wedge Y_{V},\\
& \pi_{TM,X_{C},Y_{C},c}=\pi_{TM}+\lambda c({\bf x})X_{C}\wedge Y_{C}
\end{align}
are the Poisson tensors on $TM$.
\end{theorem}
\begin{proof}
By calculation of the Schouten--Nijenhuis bracket we obtain
\begin{equation}
[\pi_{TM,X_{C},Y_{V},c},\pi_{TM,X_{C},Y_{V},c}]=2 \lambda c({\bf x})[\pi_{TM},X_{C}\wedge Y_{V}]=0,
\end{equation}
because $X_{C}$ and $Y_{V}$ are Poisson vector fields for $\pi_{TM}$.
The proof for the second bi-vector is completely analogous.
\end{proof}

In the local coordinates expressions of the Poisson structures introduced in Theorem \ref{C} are the following
\begin{equation}
\label{ten-3-c}
\pi_{TM,X_{C},Y_{V},c}  ({\bf x}, {\bf y})=
\end{equation}
$$
\begingroup\makeatletter\def\f@size{8}\check@mathfonts 
\left(
\begin{array}{c|c}
 0 & \pi ({\bf x})+ \lambda c({\bf x})v({\bf x})w^{\top}({\bf x})\\
\hline
\pi ({\bf x})- \lambda c({\bf x})v({\bf x})w^{\top}({\bf x}) & \sum_{s=1}^{N}\left(\dfrac{\partial\pi }{\partial x^s}({\bf x})+
\lambda c({\bf x})\left( \dfrac{\partial v }{\partial x^s}({\bf x})w^{\top}({\bf x})-w({\bf x})\left( \dfrac{\partial v }{\partial x^s}({\bf x})\right)^{\top}\right)\right)y^s
\end{array}
\right),
\endgroup
$$
\begin{equation}
\label{ten-3-d}
\pi_{TM,X_{C},Y_{C},c}  ({\bf x}, {\bf y})=
\end{equation}
$$\left(
\begingroup\makeatletter\def\f@size{4}\check@mathfonts 
\begin{array}{c|c}
 v({\bf x})w^{\top}({\bf x}) -w({\bf x})v^{\top}({\bf x})  & \pi ({\bf x})+ \lambda c({\bf x})\sum_{s=1}^{N}
\left( \dfrac{\partial v }{\partial x^s}({\bf x})w^{\top}({\bf x})-w({\bf x})\left( \dfrac{\partial v }{\partial x^s}({\bf x})\right)^{\top}\right)y^s\\
\hline
\pi ({\bf x})- \lambda c({\bf x})\sum_{s=1}^{N}
\left( \dfrac{\partial v }{\partial x^s}({\bf x})w^{\top}({\bf x})-w({\bf x})\left( \dfrac{\partial v }{\partial x^s}({\bf x})\right)^{\top}\right)y^s
 & \sum_{s=1}^{N}\left(\dfrac{\partial\pi }{\partial x^s}({\bf x})+
\lambda c({\bf x})\sum_{m=1}^{N}\left( \dfrac{\partial v }{\partial x^s}({\bf x})
\left( \dfrac{\partial w }{\partial x^s}({\bf x})\right)^{\top}
-\dfrac{\partial w }{\partial x^s}({\bf x})
\left( \dfrac{\partial v }{\partial x^s}({\bf x})\right)^{\top}
\right)y^m\right)y^s
\end{array}
\endgroup
\right).
$$

\begin{theorem}
Let $c_1,\dots, c_r$ be  Casimir functions for the Poisson structure $\pi$ such that $\mathcal{L}_{X}c_i=0$, $\mathcal{L}_{Y}c_i=0$. Then the functions 
\begin{equation}
\label{cas-1-1-1}
c_i\circ q_M \quad \textrm{and}\quad l_{dc_i}=\sum_{s=1}^{N}\dfrac{\partial c_i}{\partial x_s}({\bf x})y_s, \quad i=1,\dots r,
\end{equation}
are the Casimir functions for the Poisson tensors $\pi_{TM,X_{V},Y_{V},c}$, $\pi_{TM,X_{C},Y_{V},c}$ and $\pi_{TM,X_{C},Y_{C},c}$. 
\end{theorem}
\begin{proof}
This is the consequence of Theorem \ref{th-333} and formula (\ref{cas-3}).
\end{proof}
In addition, $c_i\circ q_M$ is always a Casimir function for the Poisson tensor $\pi_{TM,X_{V},Y_{V},c}$.

Those procedures can be repeated many times, with certain assumptions,
which gives the following theorems.
\begin{theorem}
\label{theorem-last}
Let $X_i\in\mathcal{X}(M)$ for $i=1,2,3,4$ and let $X_{i,C}$,  $X_{i,V}$ be the complete  and vertical lifts of the vectors 
$X_i$ on $TM$.   
\begin{enumerate}
\item If $[X_1,X_i]=0$, $[X_2,X_3]=0$, $[X_3,X_4]=0$, then\begin{align}
\label{con-4-p} & \pi_{X_{1,C}, X_{2,V},X_{3,C}, X_{4,V}}=X_{1,C}\wedge X_{2,V}+X_{3,C}\wedge X_{4,V}
\end{align}
is the Poisson tensor on $TM$.
\item If $[X_i,X_j]=0$ for $i,j=1,2,3,4$, then\begin{align}
& \pi_{X_{1,C}, X_{2,C},X_{3,C}, X_{4,V}}=X_{1,C}\wedge X_{2,C}+X_{3,C}\wedge X_{4,V},\\
& \pi_{X_{1,C}, X_{2,C},X_{3,C}, X_{4,C}}=X_{1,C}\wedge X_{2,C}+X_{3,C}\wedge X_{4,C}
\end{align}
are the Poisson tensors on $TM$.
\item If $[X_1,X_i]=0$, $[X_2,X_i]=0$, then
\begin{align}
& \pi_{X_{1,C}, X_{2,C},X_{3,V}, X_{4,V}}=X_{1,C}\wedge X_{2,C}+X_{3,V}\wedge X_{4,V}
\end{align}
is the Poisson tensor on $TM$.
\item If $[X_1,X_i]=0$,  then\begin{align}
& \label{con-4-ost} \pi_{X_{1,C}, X_{2,V},X_{3,V}, X_{4,V}}=X_{1,C}\wedge X_{2,V}+X_{3,V}\wedge X_{4,V}
\end{align}
is the Poisson tensor on $TM$.
\end{enumerate}
\end{theorem}

\begin{proof}
Our proof starts with observation that
\begin{equation}
[X\wedge Y, Z\wedge W]=[X,Z]\wedge Y\wedge W+ [Z,Y]\wedge X\wedge W+[X,W]\wedge Z\wedge Y+[W,Y]\wedge Z\wedge X
\end{equation}
for $X,Y,Z,W\in\mathcal{X}(TM)$.
Applying this equality for all above bi-vector field cases, we get our conclusions.
\end{proof}
The above theorem gives the restrictive conditions  for these structures to be bi--Hamiltonian.
Moreover we have the following corollary.

\begin{corollary}
\label{col-1}
If the  bi-vectors (\ref {con-4-p}--\ref{con-4-ost}) are Poisson tensors  and $X_i$ for $i=1,2,3,4$ are Poisson vector fields on $M$ then Poisson tensor $\pi_{TM}$ is compatible with them. 
\end{corollary}

\section{Examples}

Let us take $\mathbb{R}^3$ with local coordinates ${\bf x}=(x^1,x^2,x^3)$ and let us consider the linear Poisson structure given by the Poisson tensor 
\begin{equation}
\pi({\bf x})=x^1\frac{\partial}{\partial x^2}\wedge\frac{\partial}{\partial x^3},
\end{equation} 
which equivalently can be written in the following form
\begin{equation}\label{p31}
\pi({\bf x})=\left( \begin{array}{ccc} 0&0&0\\0&0&x^1\\0&-x^1&0\\ \end{array}\right).
\end{equation}
Linear Poisson structure given above is related to the Lie algebra $\mathcal{A}_{3,1}.$ The commutation rule for this Lie algebra is $[e_2,e_3]=e_1$ and it has only one invariant which is $e_1,$  see \cite{n1}. In this case the Casimir function for $\pi$ assumes following form 
$$c_1({\bf x})=x^1.$$
It is easy to see that a Poisson vector field in this case is given by
\begin{equation}
X=x^1\left(\frac{\partial v^2}{\partial x^2}({\bf x})+\frac{\partial v^3}{\partial x^3}({\bf x})\right)\frac{\partial}{\partial x^1}+v^2({\bf x})\frac{\partial}{\partial x^2}+v^3({\bf x})\frac{\partial}{\partial x^3},
\end{equation} 
where $\frac{\partial v^2}{\partial x^2}({\bf x})+\frac{\partial v^3}{\partial x^3}({\bf x})=f(x^1)$ and $f$ is an arbitrary function of   one variable.
We can lift the Poisson tensor on $\mathbb{R}^3$ to $T\mathbb{R}^3,$  then 
\begin{equation}
\label{ex-1}
\pi_{TM}({\bf x},{\bf y})=\left(\begin{array}{ccc|ccc} 0&0&0&0&0&0\\0&0&0&0&0&x^1\\0&0&0&0&-x^1&0\\\hline 0&0&0&0&0&0\\0&0&x^1&0&0&y^1\\0&-x^1&0&0&-y^1&0\end{array}\right),
\end{equation}
where $({\bf x},{\bf y})=(x^1,x^2,x^3,y^1,y^2,y^3).$ The Casimir functions for this structure are given by $c_1({\bf x},{\bf y})=x^1,\quad c_2({\bf x},{\bf y})=y^1$ . We can also see that this is a Lie--Poisson structure associated with Lie algebra $\mathcal{A}_{6,4},$ for which commutation rules are $[e_1,e_2]=e_5, [e_1,e_3]=e_4,[e_2,e_4]=e_6$ and $(x^1,x^2,x^3,y^1,y^2,y^3)\mapsto (e_6,-e_4,e_3,e_5,e_1,e_2)$, see \cite{AJ, n1}.

Now we present the list of some infinitesimal deformations of the Poisson tensor $\pi_{TM}$ given in (\ref{ex-1}) through the choice of different Poisson vector fields.
\begin{enumerate}
\item Let us now  take as a Poisson vector  field 
\begin{equation}
X=\sqrt{x^3}\dfrac{\partial}{\partial x^2}
\end{equation}
 and put $\lambda c(x)=1$.
Then the complete and vertical lifts are given by
\begin{align}
X_{C}=\sqrt{x^3}\dfrac{\partial}{\partial x^2}+\dfrac{y^3}{2\sqrt{x^3}}\dfrac{\partial}{\partial y^2}, & & X_{V}=\sqrt{x^3}\dfrac{\partial}{\partial y^2}
\end{align}
 The Poisson tensor described by Theorem \ref{theorem-3a}, which in local coordinates can be written as in the (\ref{ten-3}), is given by
\begin{equation}
\pi_{TM,X_{C},X_{V},c}({\bf x},{\bf y})=\left(\begin{array}{ccc|ccc}0&0&0&0&0&0\\0&0&0&0&x^3&x^1\\0&0&0&0&-x^1&0\\\hline 0&0&0&0&0&0\\0&-x^3& x^1&0&0&y^1\\0&-x^1&0&0&-y^1&0\end{array}\right),
\end{equation}
for $\lambda=1$.
By direct calculation and changing the variables we can prove that this is a tensor for the Lie--Poisson structure associated with the Lie algebra $\mathcal{A}_{6,6}$ from the classification given in \cite{n1}. Commutation relations for this Lie algebra are $[e_1,e_2]=e_6, [e_1,e_3]=e_4, [e_1,e_4]=e_5,$ and  $[e_2,e_3]=e_5,$ where $(x^1,x^2,x^3,y^1,y^2,y^3)\mapsto (e_5,-e_3,e_4,e_6,e_1,e_2).$ Moreover the Casimir functions for structure $\pi_{TM,X_{C},X_{V},c}$ are $c_1({\bf x})=x^1$ and $c_2({\bf x},{\bf y})=l_{dx^1}=y^1$ from Theorem \ref{c-1}.

\item Let us now take Poisson vector fields 
\begin{equation}
X=x^1\dfrac{\partial}{\partial x^1}+x^2\dfrac{\partial}{\partial x^2}, \quad Y=\dfrac{\partial}{\partial x^3}.
\end{equation}
 Then their vertical lifts are of the form
\begin{equation}
X_{V}=x^1\dfrac{\partial}{\partial y^1}+x^2\dfrac{\partial}{\partial y^2}, \quad Y_{V}=\dfrac{\partial}{\partial y^3}
\end{equation}
and bi-vector can be expressed as
$$X_V\wedge Y_V=x^1\dfrac{\partial}{\partial y^1}\wedge\dfrac{\partial}{\partial y^3}+x^2\dfrac{\partial}{\partial y^2}\wedge\dfrac{\partial}{\partial y^3}.$$
Then by taking as above $\lambda c({\bf x})=1$ and considering Poisson tensor $\pi$ we get, from Theorem \ref{VV}, that $\pi_{TM}+X_{V}\wedge Y_{ V}$ is a Poisson tensor and it is  given by following matrix
 \begin{equation}
\pi_{TM, X_V,Y_V,c}({\bf x},{\bf y})=\left( \begin{array}{ccc|ccc} 0&0&0&0&0&0\\0&0&0&0&0&x^1\\0&0&0&0&-x^1&0\\\hline 0&0&0&0&0&x^1\\0&0&x^1&0&0&y^1+x^2\\0&-x^1&0&-x^1&-y^1-x^2&0\\ \end{array}\right).
\end{equation}
If we take mapping $(x^1,x^2,x^3,y^1,y^2,y^3)\mapsto (2x^1,y^1+x^2,2x^3,y^1-x^2,y^2,y^3)\mapsto(e_1,e_2,e_4,e_6,e_3,e_5)$ then we can recognize that above tensor is a Poisson tensor for Lie--Poisson structure related to direct sum $\mathcal{A}_{5,5}\oplus\big< e_6\big>$
for which commutation rules are given by $[e_3,e_4]=e_1, [e_2,e_5]=e_1, [e_3,e_5]=e_2.$ Furthermore the Casimir functions are $c_1({\bf x})=x^1, c_2({\bf x},{\bf y})=y^1-x^2$.

\item Let us take now four Poisson vector fields 
\begin{align}
& X_1=x^1\dfrac{\partial}{\partial x^1}+x^3\dfrac{\partial}{\partial x^3},\quad Y_1=-\dfrac{x^2}{x^1}\dfrac{\partial}{\partial x^3},\\
& X_2=x^1\dfrac{\partial}{\partial x^2} ,\quad Y_2=\dfrac{\partial}{\partial x^3}.
\end{align}
Then  we can  lift them vertically to Poisson vector fields on $T\mathbb{R}^3$ and get
\begin{align}
& X_{1,V}=x^1\dfrac{\partial}{\partial y^1}+x^3\dfrac{\partial}{\partial y^3},&& Y_{1,V}=-\dfrac{x^2}{x^1}\dfrac{\partial}{\partial y^3},\\  
& X_{2, V}=x^1\dfrac{\partial}{\partial y^2}, && Y_{2, V}=\dfrac{\partial}{\partial y^3}.
\end{align}
Then by taking as above $\lambda c({\bf x})=1$ and considering Poisson tensor $\pi$ we get, from Theorem \ref{theorem-last-V}, that $\pi_{TM}+X_{1,V}\wedge Y_{1, V}+X_{2, V}\wedge Y_{2, V}$ is a Poisson tensor and it is of the form
 \begin{equation}
\pi_{TM, X_{1,V}, Y_{1,V}, X_{2,V}, Y_{2,V},c}({\bf x},{\bf y})=\left( \begin{array}{ccc|ccc} 0&0&0&0&0&0\\0&0&0&0&0&x^1\\0&0&0&0&-x^1&0\\\hline 0&0&0&0&0&-x^2\\0&0&-x^1&0&0&y^1+x^1\\0&-x^1&0&x^2&-y^1-x^1&0\\ \end{array}\right).
\end{equation}
It is easy to see that it is a Poisson tensor for Lie--Poisson structure related to the Lie algebra $\mathcal{A}_{6,17}$ by taking the mapping $(x^1,x^2,x^3,y^1,y^2,y^3)\mapsto (x^1,-x^2,x^3,-y^1,y^2-\frac{1}{2}x^2,y^3-\frac{1}{2}x^3)\mapsto(e_6,e_4,e_5,e_3,e_2,e_1).$ Commutation relation for Lie algebra $\mathcal{A}_{6,17}$ are $[e_1,e_2]=e_3, [e_1,e_3]=e_4,$ $[e_1,e_4]=e_6$ and $[e_2,e_5]=e_6$ and Casimir functions are $c_1({\bf x})=x^1, c_2({\bf x},{\bf y})=(x^2)^2+2y^1x^1.$

\item Let us now consider the Poisson vector fields 
\begin{equation}
X=x^3\dfrac{\partial}{\partial x^2}+x^1\dfrac{\partial}{\partial x^3}, \quad Y=\dfrac{\partial}{\partial x^2}
\end{equation}
  and let us put $\lambda c({\bf x})=1$. Then from (\ref{vec}) we get complete and vertical lifts of the vector fields on $T\mathbb{R}^3$, given by 
\begin{equation}
X_{C}=x^3\dfrac{\partial}{\partial x^2}+x^1\dfrac{\partial}{\partial x^3}+y^3\dfrac{\partial}{\partial y^2}+y^1\dfrac{\partial}{\partial y^3}\qquad Y_{V}=\dfrac{\partial}{\partial y^2}.
\end{equation}
It is easy to see that $[X,Y]=0$ is a Poisson tensor. Then from Theorem \ref{C} we get that $\pi_{TM}+X_{C}\wedge Y_{V}$ is also a Poisson tensor and it is of the form
\begin{equation}
\pi_{TM,X_{C},Y_{V},c}({\bf x},{\bf y})=\left( \begin{array}{ccc|ccc} 0&0&0&0&0&0\\0&0&0&0&x^3&x^1\\0&0&0&0&0&0\\\hline 0&0&0&0&0&0\\0&-x^3&0&0&0&0\\0&-x^1&0&0&0&0\\ \end{array}\right).
\end{equation}
By direct calculation and changing the variables we can prove that this is a tensor for a Lie--Poisson structure related to direct sum $\mathcal{A}_{5,1}\oplus \langle e_6\rangle.$ 
Commutation rules for Lie algebra $\mathcal{A}_{5,1}$ are $[e_3,e_5]=e_1$ and $[e_4,e_5]=e_2$ where $(x_1,x_2,x_3,y_1,y_2,y_3)\mapsto (e_1,e_5,e_2,e_6,-e_4,-e_3).$ Moreover the Casimir functions for structure $\pi_{TM,X_{1,TM},Y_{2,TM},c}$ are $c_1({\bf x})=x_1, c_2({\bf x})=x_3$ and $c_3({\bf x},{\bf y})=x_1y_2-x_3y_3.$

\item Let us take now four Poisson vector fields 
\begin{align}
& X_1=\dfrac{\partial}{\partial x^3},\quad Y_1=x^1\dfrac{\partial}{\partial x^2},\\
& X_3=\dfrac{\partial}{\partial x^2}, \quad Y_2=x^3\dfrac{\partial}{\partial x^2}.
\end{align}
Then from (\ref{vec}) we can lift them to Poisson vector fields on $T\mathbb{R}^3$ and get
\begin{align}
& X_{1,C}=\dfrac{\partial}{\partial x^3},&& Y_{1,V}=x^1\dfrac{\partial}{\partial y^2},\\
& X_{2,C}=\dfrac{\partial}{\partial x^2}, && Y_{2,V}=x^3\dfrac{\partial}{\partial y^2}.
\end{align}
Then by taking as above $\lambda c({\bf x})=1$ and considering Poisson tensor $\pi$ we get, from Corollary \ref{col-1}, that $\pi_{TM}+X_{1,C}\wedge Y_{1,V}+X_{2,C}\wedge Y_{2,V}$ is a Poisson tensor and it is of the form
 \begin{equation}
\pi_{TM,X_{1,C},Y_{1,V},X_{2,C},Y_{2,V},c}({\bf x},{\bf y})=\left( \begin{array}{ccc|ccc} 0&0&0&0&0&0\\0&0&0&0&x^3&x^1\\0&0&0&0&0&0\\\hline 0&0&0&0&0&0\\0&-x^3&0&0&0&y^1\\0&-x^1&0&0&-y^1&0\\ \end{array}\right).
\end{equation}
By changing the variables and direct calculation we can recognized that this is a Poisson tensor for a Lie--Poisson structure associated with the Lie algebra $\mathcal{A}_{6,3}.$ From \cite{n1}, commutation rules for this Lie algebra are $[e_1,e_2]=e_6, [e_1,e_3]=e_4$ and $[e_2,e_3]=e_5,$ where $(x^1,x^2,x^3,y^1,y^2,y^3)\mapsto (e_4,e_1,e_6,e_5,e_2,e_3)$ and Casimir functions are $c_1({\bf x})=x^1$, $c_2({\bf x})=x^3$, $c_3({\bf y})=y^1$, $c_4({\bf x},{\bf y})=x^2y^1+x^3y^3-x^1y^2$.

\end{enumerate}

\section*{Acknowledgments}

This article has received financial support from the Polish Ministry of Science and Higher Education under subsidy for maintaining the research potential of the Faculty of Mathematics and Informatics, University of Bialystok (BST-148).
 

\bibliographystyle{plain}

\end{document}